\documentclass[12pt]{amsart}
\usepackage{graphicx, amssymb, amsthm, amsfonts, latexsym, epsfig, amsmath, amscd, amsbsy}

\theoremstyle{definition}
\newtheorem{theorem}{Theorem}[section]
\newtheorem{lem}[theorem]{Lemma}
\newtheorem{cor}[theorem]{Corollary}
\newtheorem{pro}[theorem]{Proposition}

\newtheorem{rem}[theorem]{Remark}
\newtheorem{ex}[theorem]{Example}

\numberwithin{equation}{section}

\def\Q{{\mathbb Q}}
\def\R{{\mathbb R}}
\def\N{{\mathbb N}}
\def\Z{{\mathbb Z}}
\def\T{{\mathbb T}}
\def\C0{{\mathfrak C}}

\def\dl{\underrightarrow\lim}
\def\il{\underleftarrow\lim}

\begin{document}

\title{Pure discrete spectrum in substitution tiling spaces}
\author{M. Barge}
\address{Department of Mathematics, Montana State University,
Bozeman, MT 59717, USA}
\email{barge@math.montana.edu}
\author{S. \v{S}timac}
\address{Department of Mathematics, University of Zagreb,
Bijeni\v cka 30, 10 000 Zagreb, Croatia}
\email{sonja@math.hr}
\author{R. F.  Williams}
\address{Department of Mathematics, University of Texas,
Austin, TX 78712, USA}
\email{bob@math.utexas.edu}
\thanks{S\v{S} was supported in part by the MZOS Grant 037-0372791-2802 of the Republic of Croatia.}

\begin{abstract}
We introduce a procedure for establishing pure discrete spectrum for substitution tiling systems of Pisot family type and illustrate with several examples.
\end{abstract}

\maketitle

\section{Introduction}\label{intro}

The study of Aperiodic Order received impetus in 1985 with the discovery of physical quasicrystals.
These materials have enough long range order to produce a pure point X-ray diffraction spectrum
yet they lack the strict periodicity of traditional crystals. Quasicrystals are modeled by self-affine tilings of Euclidean space and these tilings are effectively studied by means of an associated topological dynamical system, the tiling space. It has been established, through a beautiful circle of ideas and in considerable generality, that a pattern of points (atoms) produces pure point X-ray diffraction if and only if the associated tiling space has pure discrete dynamical spectrum (\cite{D},\cite{LMS}).

We introduce in this article a procedure for determining whether (or not) a tiling dynamical system has pure discrete spectrum. The procedure applies to tiling systems that arise from substitutions.
In dimension one, a necessary condition for a substitution tiling system to have pure discrete spectrum is that the expansion factor of the substitution be a Pisot number.
Recent results (\cite{LS}) show that if a substitution tiling system in any dimension has pure discrete spectrum, then the linear expansion associated with the substitution must be generally `Pisot' in nature and our technique is, correspondingly, restricted to so-called Pisot family substitutions (see the next section for definitions).

Our approach has its origins in the balanced pair algorithm for symbolic substitutive systems initiated by Dekking (\cite{Dek}). A geometric version of balanced pairs (overlap coincidences) was
introduced by Solomyak (\cite{S3}) and considered for one-dimensional substitution tiling systems in \cite{IR}, \cite{AI}, \cite{BK}, \cite{SS}, \cite{ST}, and \cite{Fo}, and for higher dimensions in \cite{L}, \cite{LM}, \cite{LMS}, and \cite{FS}. Akiyama and Lee (\cite{AL}) have implemented overlap coincidence in computer language with an algorithm that will decide whether or not a particular substitution tiling system (with the Meyer property)
has pure discrete spectrum.

Proposition 17.4 of \cite{BK} (generalized to the `reducible' case in \cite{BBK}) states that it suffices to check the convergence
of the balanced pair algorithm on a single balanced pair in order to determine pure discrete spectrum for a one-dimensional Pisot substitution tiling system. The main result of this article extends that result to Pisot family tiling spaces in all dimensions. (The proof of Theorem 16.3 of \cite{BK}, from which the one-dimensional result follows, contains a gap that is not easily plugged by the techniques of that paper. We recover the one-dimensional case here as Corollary \ref{(ab,ba)}.)

In the next section we briefly review the relevant definitions and background for substitution tiling spaces. Two subsequent sections establish sufficient and necessary conditions for pure discrete spectrum and a final section provides examples in which pure discrete spectrum is established.

\section{Definitions and background}\label{defs}

The basic ingredients for an $n$-dimensional substitution are a collection $\{\rho_1,\ldots,\rho_l\}$
of prototiles, a linear expansion $\Lambda:\R^n\to\R^n$, and a substitution rule $\Phi$. Each
{\em prototile} $\rho_i$ is an ordered pair $\rho_i=(C_i,i)$ in which $C_i$ is a compact and topologically regular ($C_i=cl(int(C_i))$) subset of $\R^n$. A translate $\rho_i+v:=(C_i+v,i)$ of a prototile by
a $v\in\R^n$ is called a {\em tile}. The {\em support of a tile} $\tau=\rho_i+v$ is the set $spt(\tau):=C_i+v$ and the {\em interior} of $\tau$ is the interior of its support: $\mathring{\tau}:=int(spt(\tau))$.  A collection of tiles with pairwise disjoint interiors is called a {\em patch}, the {\em support of a patch} $P$ is the union of the supports of its constituent tiles, $spt(P):=\cup_{\tau\in P}spt(\tau)$, and the {\em interior of a patch} $P$ is $\mathring{P}:=\cup_{\tau\in P}(\mathring\tau)$ (or sometimes $int(P)$). A {\em tiling} is a patch whose support is all of $\R^n$.

Given a collection $\mathcal{A}=\{\rho_1,\ldots,\rho_l\}$ of prototiles, let $\mathcal{A}^*$ be the collection of all finite patches consisting of collections of translates of the prototiles. A {\em substitution} with linear expansion $\Lambda$ is a map $\Phi:\mathcal{A}\to\mathcal{A}^*$ with the property: $spt(\Phi(\rho_i))=\Lambda spt(\rho_i)$ for each $i$. Such a $\Phi$ extends to a map from tiles to patches by $\Phi(\rho_i+v):=\Phi(\rho_i)+\Lambda v:=\{\tau+\Lambda v:\tau\in\Phi(\rho_i)\}$ and to a map from patches to patches
by $\Phi(P):=\cup_{\tau\in P}\Phi(\tau)$.

A substitution $\Phi$, as above, is {\em primitive} if there is $k\in\N$ so that $\Phi^k(\rho_i)$ contains a translate of $\rho_j$ for each $i,j\in\{1,\ldots,l\}$. That is equivalent to saying that
$M^k$ has strictly positive entries, where $M=(m_{i,j})$ is the $l\times l$ {\em transition matrix} for $\Phi$:
$m_{i,j}$ is the number of translates of $\rho_i$ that occur in $\Phi(\rho_j)$. A patch $P$ is {\em allowed} for $\Phi$ if there is $k\in\N$ and $v\in\R^n$ so that
$P\subset \Phi^k(\rho_i)+v$ for some $i$. The {\em tiling space} associated with a primitive substitution $\Phi$ is the set $\Omega_{\Phi}:=\{T:T$ is a tiling and all finite patches in $T$ are allowed for $\Phi\}$ with a topology in which two tilings are close if a small translate of one agrees with the other in a large neighborhood of the origin. Let us be more precise about the topology. Given a tiling $T$, let $B_0[T]:=\{\tau\in T:0\in spt(\tau)\}$ and, for $r>0$, let $B_r[T]:=\{\tau\in T:spt(\tau)\cap B_r(0)\ne\emptyset\}$. A neighborhood base at $T\in\Omega_{\Phi}$ is then
$\{U_{\epsilon}\}_{\epsilon>0}$ with $U_{\epsilon}:=\{T'\in\Omega_{\Phi}:B_{1/\epsilon}[T']=B_{1/\epsilon}[T]-v \text{ for some }v\in B_{\epsilon}(0)\}$. This is a metrizable topology (see, for example,
 \cite{AP}) and we denote by $d$ any metric that induces this topology.
The map $T\mapsto \Phi(T)$ and the $\R^n$-action $(T,v)\mapsto T-v$ are continuous and these
intertwine by $\Phi(T-v)=\Phi(T)-\Lambda v$.

A tiling $T$ has (translational) {\em finite local complexity} (FLC) if, for each $r\ge 0$, there are only finitely many patches, up to translation, of the form $B_r[T-v]$, $v\in\R^n$. The substitution $\Phi$ is said to have finite local complexity if each of its admissible tilings has finite local complexity. Finally, $\Phi$ is {\em aperiodic} if, for $T\in\Omega_{\Phi}$ and $v\in\R^n$, $T-v=T\implies v=0$. Under the assumption that $\Phi$ is primitive, aperiodic, and has finite local complexity, the following facts are by now well known (\cite{AP},\cite{S1}):
\begin{itemize}
\item $\Omega_{\Phi}$ is a continuum (i.e., a compact and connected metrizable space),
\item $\Phi:\Omega_{\Phi}\to\Omega_{\Phi}$ is a homeomorphism, and
\item The $\R^n$-action on $\Omega_{\Phi}$ is strictly ergodic (i.e., minimal and uniquely ergodic).
\end{itemize}

All substitutions $\Phi$ in this article will be assumed to be primitive, aperiodic, and to have finite local complexity. We will denote by $\mu$ the unique translation invariant Borel probability measure on $\Omega_{\Phi}$. A function $f\in L^2(\Omega_{\Phi},\mu)$ is an {\em eigenfunction}
of the $\R^n$-action with {\em eigenvalue} $b\in\R^n$ if $f(T-v)=e^{2\pi i\langle b,v\rangle}f(T)$ for
all $v\in\R^n$, and almost all $T\in\Omega_{\Phi}$. The $\R^n$-action on $\Omega_{\Phi}$ is said to have {\em pure discrete spectrum} if the linear span of the eigenfunctions is dense in
$L^2(\Omega_{\Phi},\mu)$. Solomyak proves in \cite{S2} that every eigenfunction is almost everywhere equal to a continuous eigenfunction. It is then a consequence of the Halmos - von Neumann theory that the $\R^n$-action on $\Omega_{\Phi}$ has pure discrete spectrum if and only if there is a compact abelian group $X$, an $\R^n$-action on $X$ by translation, and a continuous
semi-conjugacy $g:\Omega_{\Phi}\to X$ of $\R^n$-actions that is a.e. one-to-one (with respect to Haar measure).

There are certain necessary conditions on the linear expansion $\Lambda$ for the substitution $\Phi$ in order for the $\R^n$-action on $\Omega_{\Phi}$ to have pure discrete spectrum (pds). For there to even be a substitution associated with $\Lambda$, it's necessary that the eigenvalues of $\Lambda$ be algebraic integers (\cite{K},\cite{LS1}). For pds, the eigenvalues of $\Lambda$ must constitute a {\em Pisot family}: If $\lambda$ is an eigenvalue of $\Lambda$ and $\lambda'$ is an algebraic conjugate of $\lambda$ with $|\lambda'|\ge 1$, then $\lambda'$ is also an eigenvalue of $\Lambda$ and of the same multiplicity as $\lambda$ (\cite{LS}). Under the additional assumptions that $\Lambda$ is diagonalizable over $\mathbb{C}$ and all eigenvalues of $\Lambda$ are algebraic conjugates - say, of degree $d$ - and of the same multiplicity - say $m$ - Lee and Solomyak prove (\cite{LS}) that the $\R^n$-action on $\Omega_{\Phi}$ has a generous compliment of eigenfunctions. We will say that a substitution with a linear expansion whose eigenvalues are as in the previous sentence is of {\em $(m,d)$-Pisot family type} (or, if $m$ and $d$ are unimportant, just of {\em Pisot family type} - we'll also say that $\Omega=\Omega_{\Phi}$ is a {\em Pisot family tiling space}). It is proved in
\cite{BKe} that if $\Phi$ is of $(m,d)$-Pisot family type, then there is a continuous finite-to-one semi-conjugacy $g:\Omega_{\Phi}\to X$ of the $\R^n$-action on $\Omega_{\Phi}$ with an action of $\R^n$ by translation on a group $X$. Moreover, there is a hyperbolic endomorphism $F:\T^{md}\to\T^{md}$ of the $md$-dimensional torus so that $X$ is the inverse limit $X=\hat{\T}^{md}:=\il F$.
The $\R^n$-action on $\hat{\T}^{md}$ is a Kronecker action by translation along the $n$-dimensional unstable leaves of the shift homeomorphism $\hat{F}:\hat{\T}^{md}\to\hat{\T}^{md}$ and, in addition, $g\circ \Phi=\hat{F}\circ g$. (If the eigenvalues of $\Lambda$ are algebraic units then $F$ is an automorphism and $\hat{\T}^{md}$ is simply the $md$-dimensional torus.) The group
$\hat{\T}^{md}$, with the Kronecker action, is the {\em maximal equicontinuous factor} of the $\R^n$-action on $\Omega_{\Phi}$: If there is a semi-conjugacy $g':\Omega_{\Phi}\to Y$ with an equicontinuous $\R^n$-action on $Y$, then there is a semi-conjugacy $g'':\hat{\T}^{md}\to Y$ with
$g'=g''\circ g$.

Tilings $T,T'\in\Omega$ are {\em regionally proximal}, $T\sim_{rp}T'$, provided there are $T_k,T'_k\in\Omega$ and $v_k\in\R^n$ so that $d(T,T_k)\to0$, $d(T',T'_k)\to0$, and $d(T_k-v_k,T'_k-v_k)\to0$ as $k\to\infty$. It is proved in \cite{BKe} that if $\Omega$ is a Pisot family tiling space, then $T$ and $T'$ in $\Omega$ are regionally proximal if and only if $T$ and $T'$ are {\em strongly regionally proximal}, $T\sim_{srp}T'$: for each $k\in\N$ there are $T_k,T'_k\in\Omega$ and $v_k\in\R^n$ with $B_k[T_k]=B_k[T]$,  $B_k[T'_k]=B_k[T']$, and $B_k[T_k-v_k]=B_k[T'_k-v_k]$. The notion of regional proximality extends in an obvious way to arbitrary group actions and Auslander proves in \cite{Aus} that, in the case of minimal abelian actions on compact metric spaces, regional proximality is the equicontinuous structure relation.

\begin{theorem}(Auslander) If $G$ is an abelian group acting minimally on the compact metric space $X$, and $g:X\to X_{max}$ is the maximal equicontinuous factor, then $g(x)=g(x')$ if and only if $x\sim_{rp}x'$.
\end{theorem}

We thus have:

\begin{cor}\label{srp} If $\Omega$ is a Pisot family tiling space with maximal equicontinuous factor
$g:\Omega\to\hat{\mathbb{T}}^{md}$, then $g(T)=g(T')$ if and only if $T\sim_{srp}T'$.
\end{cor}

\section{Sufficient conditions for pure discrete spectrum}\label{method}

Given patches $Q,Q'$ (not necessarily allowed for $\Phi$) and  $x\in int(spt(Q)\cap spt(Q'))$, we'll say that $Q$ and $Q'$ are {\em coincident}
at $x$ provided $B_0[Q-x]=B_0[Q'-x]$ and {\em eventually coincident} at $x$ if there is $k\in\N$ so that $B_0[\Phi^k(Q-x)]=B_0[\Phi^k(Q'-x)]$.
We'll say that $Q$ and $Q'$ are {\em densely eventually coincident on overlap} if $Q$ and $Q'$ are eventually coincident at a set of $x$ that is dense in $int(spt(Q)\cap spt(Q'))$, and, if this happens for tilings $Q,Q'$, we'll say that $Q$ and $Q'$ are {\em densely eventually coincident}. If $Q$ is a finite patch and there is a basis $\{v_1,\ldots,v_n\}$ of $\R^n$ so that $\cup_{k_1,\ldots,k_n\in\Z}(Q+\sum k_iv_i)$ is a tiling of $\R^n$, we'll denote this tiling by $\bar{Q}$. A vector $v=\sum_{i=1}^na_iv_i$ is {\em completely rationally independent} of the basis $\{v_1,\ldots,v_n\}$ if the numbers $a_1,\ldots,a_n$ are independent over $\Q$. This is the same thing as saying that the map
$x+L\mapsto x+v+L$ is a minimal homeomorphism of the torus $\R^n/L$,
$L$ being the lattice spanned by $\{v_1,\ldots,v_n\}$ over $\Z$.

\begin{theorem} \label{main theorem} Suppose that $\Phi$ is a substitution of Pisot family type and that $Q$ is a finite patch such that $\bar{Q}=\cup_{k_1,\ldots,k_n\in\Z}(Q+\sum k_iv_i)$ is a tiling of $\R^n$. Suppose also that $v\in\R^n$ is completely rationally independent of $\{v_1,\ldots,v_n\}$ and that $\bar{Q}$ and $\bar{Q}-v$ are densely eventually coincident. Then the $\R^n$-action on $\Omega_{\Phi}$ has pure discrete spectrum.
\end{theorem}
\begin{proof}
 Let $L$ be the lattice $L:=\{\sum k_iv_i:k_i\in\Z\}$ and let $D\subset spt(Q)$ be a fundamental domain for $\R^n/L$. For $x,y\in\R^n$, by $x\oplus y=z$ we will mean that $z\in D$ and $x+y-z\in L$. Note that eventual coincidence at a point is an open condition, is translation equivariant, and only depends on local structure: if $Q_1$ is eventually coincident with $Q'_1$ at $x$, $B_0[Q_2-y]=B_0[Q_1]$, and $B_0[Q'_2-y]=B_0[Q'_1]$, then $Q_2$ is coincident with $Q'_2$ at
$x+y$.
We first prove that if $x,w\in\R^n$ are such that $\bar{Q}$ and $\bar{Q}-w$ are eventually coincident at $x$, then
$\bar{Q}$ and $\bar{Q}-w$ are densely eventually coincident. To see that this is the case, we may suppose that $x\in D$. Now there is a neighborhood $U$ of $x$ in $D$ so that $\bar{Q}$ and $\bar{Q}-w$ are eventually coincident at $y$ for all $y\in U$. With $v$ as hypothesized, there is a dense open subset $U_1$ of $U$ so that $\bar{Q}$ is eventually coincident with $\bar{Q}-v$ at each $y\in U_1$. Let us introduce the notation $y\sim_c y'$ to mean $\bar{Q}-y$ is coincident with $\bar{Q}-y'$ at $0$. Note that $\sim_c$ is an equivalence relation and if $y\sim_c y+z$ then $y\sim_c y\oplus z$. We have
$y\sim_c y\oplus v$ and $y\sim_c y\oplus w$ for all $y\in U_1$. Now, there is a dense open subset $U_2$ of $U_1$ so that $y\oplus w\sim_c y\oplus w\oplus v$ for all $y\in U_2$ (there is a dense open subset $U'_2$ of $U_1\oplus w$ with $y'\sim_cy'+v$ for all $y'\in U'_2$; let $U_2\subset U_1$ satisfy $U_2\oplus w=U'_2$). For $y\in U_2$,
$y\sim_cy\oplus w$, and $(y\oplus v)\sim_c (y\oplus v)\oplus w$. Inductively, we construct dense open subsets $U_k\subset U_{k-1}\subset\cdots\subset U_1$ of $U$ so that if $y\in U_k$, then
$(y\oplus jv)\sim_c(y\oplus jv)\oplus w$ for $j=1,\ldots,k$. Now pick $y\in\cap_kU_k$ (such exists
by the Baire Category Theorem). The assumption of the independence of $v$ from the $v_i$ guarantees
that $\{y\oplus jv:j\in\N\}$ is dense in $D$. Thus $\bar{Q}$ and $\bar{Q}-w$ are densely eventually coincident.

If $\bar{Q}$ and $\bar{Q}-v$ are eventually coincident at $x$, then $\Phi^k(\bar{Q})=\bar {Q}_k$ and $\Phi^k(\bar{Q}-v)=\bar{Q}_k-\Lambda^kv$, where $Q_k:=\Phi^k(Q)$, are eventually coincident at $\Lambda^kx$, for all $k\in\N$. Also, $\Lambda^kv$ is completely rationally independent of $\{\Lambda^kv_1,\ldots,\Lambda^kv_n\}$. Thus, given $R$, we may assume (by replacing $Q$ with $Q_k$ for sufficiently large $k$) that $Q$ contains every
allowed patch of diameter less than $R$. From this, we see that if $S$ and $S'$ are any two allowed finite patches that share a tile, then $S$ and $S'$ are densely eventually coincident on overlap. Indeed, we may assume that $Q$ contains translated copies of both $S$ and $S'$. Let $S-u\subset Q$ and let $w$ be such that $Q-w\supset S'-u$. There is then $x\in spt(Q)$ so that $Q$ and $Q-w$ are coincident at $x$ (since $Q$ and $Q-w$ agree on the overlap of $S-u$ and $S'-u$). From the above, $\bar{Q}$ and $\bar{Q}-w$ are densely eventually coincident; in particular, $S-u$ and $S'-u$, hence $S$ and $S'$, are densely eventually coincident on overlap.

Now, if the $\R^n$-action on $\Omega$ does not have pure discrete spectrum, then for each $T\in\Omega$ there are finitely many, and at least two, $T'\in\Omega$ so that $T$ and $T'$ are regionally proximal and $T$ and $T'$ don't share a tile (see \cite{BKe}). Pick $T\in\Omega$ that is $\Phi$-periodic. There is then $T'\in\Omega$ that is also $\Phi$-periodic, that is strongly regionally proximal with $T$ (Corollary \ref{srp}), but disjoint from $T$. There are then $T_0,T'_0\in\Omega$ and $v_0\in\R^n$ so that
$B_0[T_0]=B_0[T]$, $B_0[T'_0]=B_0[T']$ and $B_0[T_0-v_0]=B_0[T'_0-v_0]$.
Let $S$ and $S'$ be the finite patches in $T_0$ and $T'_0$: $S:=B_0[T_0]\cup B_0[T_1-v_0]+v_0$ and $S':=B_0[T'_0]\cup B_0[T'_0-v_0]+v_0$. Then $S$ and $S'$ are densely eventually coincident on overlap. In particular, there is $x$ at which $T$ and $T'$ are eventually coincident. But $T$ and $T'$ are periodic, so eventual coincidence means that $T$ and $T'$ share a tile, contrary to assumption. Thus, the $\R^n$-action on $\Omega$ must have pure discrete spectrum.
\end{proof}

Suppose now that $\phi:\mathcal{A}\to\mathcal{A}^*$, $\mathcal{A}=\{1,\ldots,k\}$, is a substitution on $k$ letters. The {\em incidence matrix} of $\phi$ is the $k\times k$ matrix $M=M_{\phi}$ whose
$i,j$-th entry is the number of $i$'s that appear in $\phi(j)$. The substitution is {\em primitive} if some power of $M$ has all positive entries; in this case there is a positive left eigenvector $\omega=(\omega_1,\ldots,\omega_k)$ corresponding to the Perron-Frobenius eigenvalue $\lambda$ of $\phi$. We'll say that $\phi$ is a {\em Pisot substitution} if $\phi$ is primitive and $\lambda$ is a Pisot number. Such a substitution on letters determines a geometrical substitution $\Phi$ on
prototiles $\rho_i:=[0,\omega_i]$, $i=1,\ldots,k$, that is of $(1,d)$ - Pisot family type, $d$ being the algebraic degree of $\lambda$.

Given a word $u\in\mathcal{A}^*$, the {\em abelianization} of $u$ is the vector $[u]\in\Z^k$ with $i$-th entry equal to the number of $i$'s that appear in $u$. A pair of words $(u,v)\in\mathcal{A}^*\times\mathcal{A}^*$ is a {\em balanced pair} if $[u]=[v]$. If $(u_1,v_1)$ and $(u_2,v_2)$ are balanced pairs, we will write
$(u_1,v_1)(u_2,v_2)=(u_1u_2,v_1v_2)$ and call $(u_1,v_1)$ and $(u_2,v_2)$ {\em factors} of $(u_1u_2,v_1v_2)$. A balanced pair $(u,v)$ is {\em irreducible} if it has no (non-trivial) factors. Clearly, every balanced pair $(u,v)$ can be written (uniquely) as a product of irreducible balanced pairs: each of these irreducible balanced pairs is called an {\em irreducible factor} of $(u,v)$. A balanced pair of the form $(a,a)$, $a\in\mathcal{A}$, is called a {\em coincidence}. Given a balanced pair $(u,v)$, we will say that {\em the balanced pair algorithm for $(u,v)$ terminates with coincidence} provided: (1) the collection $\{(x,y):
(x,y)$ is an irreducible factor of $(\phi^m(u),\phi^m(v))$ for some $m\in\N\}$ is finite; and (2) if $(x,y)$ is an irreducible factor of $(\phi^m(u),\phi^m(v))$ for some $m\in\N$, then there is $k\in\N$ so that
$(\phi^k(x),\phi^k(y))$ has a coincidence as a factor.

\begin{theorem}\label{uv,vu} Suppose that $\phi:\mathcal{A}\to\mathcal{A}^*$ is a Pisot substitution with left Perron-Frobenius eigenvector $\omega$ and suppose that $u,v\in\mathcal{A}^*$ are words with the properties: $\langle[u],\omega\rangle$ and $\langle [v],\omega\rangle$ are independent over $\Q$; and the balanced pair algorithm for $(uv,vu)$ terminates with coincidence. Then the $\R$-action on
$\Omega_{\Phi}$ has pure discrete spectrum.
\end{theorem}
\begin{proof} Let $Q$ be a patch with underlying word $uv$. The length of the support of $Q$ is $l:=\langle [u]+[v],\omega\rangle$ and $\bar{Q}=\cup_{n\in\Z}(Q+nl )$ is a periodic tiling of $\R$.
Let $l_u:=\langle [u],\omega\rangle$. The hypotheses imply
that $l_u$ is completely rationally independent of $\{l\}$ and that $\bar{Q}$ and $\bar{Q}-l_u$ are densely eventually coincident. The result follows from Theorem \ref{main theorem}.
\end{proof}

The following corollary appears in \cite{BK}. The proof given there has a gap that is not easily fixed using the techniques of that paper.

\begin{cor} \label{(ab,ba)} Suppose that $\phi$ is a Pisot substitution whose incidence matrix has characteristic polynomial irreducible over $Q$. If there are letters $a\ne b$ in the alphabet for $\phi$ so that the balanced pair algorithm for $(ab,ba)$ terminates with coincidence, then the $\R$-action on $\Omega_{\Phi}$ has pure discrete spectrum.
\end{cor}
\begin{proof} Irreducibility of the incidence matrix implies that $\langle [a],\omega\rangle$ and $\langle [b],\omega\rangle$ are independent over $\Q$.
\end{proof}

\begin{rem} It may appear that the applicability of the methods presented above depends on the geometry of the tiles: as stated, Theorem \ref{main theorem} requires the existence of a patch that tiles $\R^n$ periodically. But  in fact if $P$ is any patch and $\{v_1,\ldots,v_n\}$ is any basis with
$\sigma:=\{\sum_{i=1}^nt_iv_i:t_i\in[0,1]\}\subset spt(P)$ then one can consider the generalized patch $Q$ consisting of the generalized tiles $(spt(\tau)\cap\sigma,\tau)$, for $\tau\in P$ with $int(spt(\tau))\cap\sigma\ne\emptyset$. It makes perfectly good sense to speak of $\bar{Q}$ and $\bar{Q}-v$ being densely eventually coincident and the proof of Theorem \ref{main theorem} goes through without change.
\end{rem}

\section{Necessary conditions for pure discrete spectrum}\label{converse}

Throughout this section $\Phi$ will be an $n$-dimensional substitution of $(m,d)$-Pisot family type
with tiling space $\Omega=\Omega_{\Phi}$, inflation $\Lambda$, and maximal equicontinuous factor $g:\Omega\to\hat{\mathbb{T}}^{md}$. A vector $v\in\R^n$ is a {\em return vector} for $\Phi$ if
for some (equivalently, every) $T\in\Omega$, $T-v\cap T\ne\emptyset$. Let $\mathcal{R}=\mathcal{R}_{\Phi}:=\{v:v \text{ is a return vector for }\Phi\}$ denote the set of return vectors for $\Phi$.

We address the question: If the $\R^n$-action on $\Omega$ has pure discrete spectrum, for what $v\in\R^n$ and what tilings $S$
must it be the case that $S$ and $S-v$ are densely eventually coincident? It is clear (if $S\in\Omega$) that the collection of such $v$ is countable: if $S$ and $S-v$ are densely eventually coincident, then $\Phi^k(S)$ and $\Phi^k(S-v)=\Phi^k(S)-\Lambda^kv$ share a tile for some $k\in\N$ so that $\Lambda^kv$ is a return vector, of which there are only countably many by FLC. We will see below that for the class of all tilings in $\Omega$, the answer to the above question is precisely $\cup_{k\in\Z}\Lambda^k\mathcal{R}$. Under additional conditions this collection is a group and we will give it a homological interpretation.

For ease of exposition, we assume that the prototiles $\rho_i$ of $\Phi$ are polytopes and that the tiles in tilings in $\Omega$ meet full face to full face in all dimensions (so each $\rho_i$ has finitely many faces in each dimension, and if $\tau,\sigma\in T\in\Omega$ meet in a point $x$ that is in the relative interior of a face of $\tau$ or $\sigma$, then they meet in that entire face). We may assume that the prototiles `force the border' - otherwise, replace the prototiles by the collared prototiles (see \cite{AP}). Let $X$ be the Anderson-Putnam complex for $\Phi$:  $X$ is a cell complex with one n-cell for each (collared) prototile and these n-cells are glued along faces by translation. If $k$-face $f$ of $\rho_i$ is glued to $k$-face $g$ of $\rho_j$, then there is $v\in\R^n$ so that $f=g-v$, and the gluing is defined by $f\ni x\sim x+v\in g$. The rule for deciding what faces are glued is as follows. The $k$-face $f$ of $\rho_i$ is {\em adjacent} to the $k$-face $g$ of $\rho_j$ if there are $T\in\Omega$ and $u,v\in\R^n$ so that $\rho_i+u,\rho_j+v\in T$ and $\rho_i+u\cap\rho_j+v=f+u=g+v$. Let {\em gluable} be the smallest equivalence relation containing the adjacency relation. Two $k$-faces of prototiles are glued in $X$ if they are gluable. The substitution $\Phi$ on the prototiles induces a continuous map $f=f_{\Phi}:X\to X$ with the property that the inverse limit $\il f$ is homeomorphic with $\Omega$ by a homeomorphism that conjugates the shift $\hat{f}$ on $\il f$ with $\Phi$ on $\Omega$ (\cite{AP}).

We say that a patch $Q$ is {\em admissible} if whenever tiles $\tau,\sigma\in Q$ meet along $k$-faces, those $k$-faces are gluable.
Certainly every patch that is allowed for $\Phi$ is also admissible and if $Q$ is an admissible patch, then so is $\Phi(Q)$. For each admissible patch $Q$ there is then a continuous map $p^Q:spt(Q)\to X$
defined by $p^Q(x)=[x-v]$ provided $x\in spt(\rho_i)+v$ and $\rho_i+v\in Q$. Furthermore, if $f:X\to X$ is the continuous map induced by $\Phi$, then $f\circ p^Q=p^{\Phi(Q)}\circ\Lambda$ on $spt(Q)$ for all admissible patches $Q$.

Now given a path $\gamma:[a,b]\to X$ and $x\in\R^n$, there is a unique (continuous) path $\bar{\gamma}:[a,b]\to\R^n$ with $\bar{\gamma}(a)=x$ and $p^Q(\bar{\gamma}(t))=\gamma(t)$ whenever
$\bar{\gamma}(t)\in spt(Q)$ for any admissible patch $Q$. We will call such a $\bar{\gamma}$ a {\em lift} of $\gamma$ to $\R^n$. If $\gamma$ is a loop in $X$ then $\gamma$ is also a singular 1-cycle and we denote its homology class by $[\gamma]\in H:=H_1(X;\Z)$. If $T\in\Omega$ and $\tau, \tau+v\in T$ (thus $v$ is a return vector), there is a corresponding element $[\gamma]\in H$, where
$\gamma(t):=p^T(x+tv)$ for fixed $x\in spt(\tau)$ and $t\in [0,1]$. Note that $[\gamma]$ doesn't depend on the choice of $x\in spt(\tau)$, but it does depend on $\tau$ (not just $v$). We will show that the reverse of this process is better behaved.

Given a path $\gamma:[a,b]\to X$, let $\bar{\gamma}$ be a lift of $\gamma$ to $\R^n$ and let $l(\gamma):=\bar{\gamma}(b)-\bar{\gamma}(a)$. This is well-defined, i.e., independent of lift, since all lifts
of $\gamma$ are translates of one another (as is easily seen by uniqueness). It is also clear that if $\gamma=\gamma_1*\gamma_2$ is a concatenation of two paths, then $l(\gamma)=l(\gamma_1)+l(\gamma_2)$ and if $\gamma^{-1}$ is the reverse of $\gamma$, then $l(\gamma^{-1})=-l(\gamma)$.
By lifting homotopies, one sees that $l$ induces a homomorphism from the fundamental group of $X$ to the additive group
$\R^n$. We will see that $l$ also passes to homology.

\begin{lem} If $\gamma$ is a loop in $X$ and $0=[\gamma]\in H$ then $l(\gamma)=0$.
\end{lem}
\begin{proof} By the above remark on concatenations, we may assume that $\gamma(0)$ is a vertex in $X$. Suppose first that $n=1$. We may then homotope $\gamma$ to a loop that doesn't turn in the interior of any edge and this can be done without affecting $l$. We may assume that $\gamma$ itself has this property. Now $\gamma=\gamma_1*\cdots*\gamma_k$ is a concatenation of paths, each running along a single edge exactly once. If $\gamma_i$ and $\gamma_j$ run along the same edge, but in different directions, then $l(\gamma_i)=-l(\gamma_j)$. Since $[\gamma]=0$, each $\gamma_i$ can be paired with a $\gamma_j$ so that elements of a pair traverse the same edge, but in opposite directions.
Thus $l(\gamma)=\sum l(\gamma_i)=0$.

Suppose now that $n>1$. After sufficient simplicial subdivision of $X$, we may homotope $\gamma$ to a loop
that lies in the 1-skeleton of this subdivision, only self-intersects at vertices of this subdivision, and that has a continuous lift that also runs from $\bar{\gamma}(0)$ to $\bar{\gamma}(1)$. Let us replace $\gamma$ with such a curve.
Let $\Gamma=\sum e_i$ be the simplicial 1-chain representing $\gamma$: the $e_i$ are the oriented 1-simplices in the image of $\gamma$ and we have modified $\gamma$ so that $e_i\ne\pm e_j$ for $i\ne j$. Since $\gamma$ is homologous to zero in $X$, there is a simplicial 2-chain $\Sigma=\sum n_j\sigma_j$ in the subdivided  simplicial structure of $X$ with $\partial \Sigma=\Gamma$. After further subdivision, we may assume that each $\sigma_j$ has at most one edge on $\Gamma$.

Now there is a lift of $\Gamma$ to a 1-chain $\bar{\Gamma}$ in $\R^n$ with $\partial \bar{\Gamma}=\bar{\gamma}(1)-\bar{\gamma}(0)$.  We may lift $\Sigma$ to a 2-chain $\bar{\Sigma}=\sum n_j\bar{\sigma}_j$ in $\R^n$  in which each $\bar{\sigma}_j$ is a lift of $\sigma_j$ and $\bar{\sigma}_j$ has an edge on $\bar{\Gamma}$ if and only if $\sigma_j$ has an edge on $\Gamma$ (that is,  $\sigma_j$ meets $\Gamma$ along $e_i$ if and only if $\bar{\sigma}_j$ meets $\bar{\Gamma}$ along $\bar{e}_i$).

Now $\partial(\bar{\Sigma})=\bar{\Gamma}+\bar{\Delta}$ with none of the elementary 1-simplices of $\bar{\Delta}$ in $\bar{\Gamma}$. Since $\partial^2(\bar{\Sigma})=0$, $$\partial(\bar{\Delta})=\bar{\gamma}(0)-\bar{\gamma}(1).$$
Each edge of $\bar{\Delta}$ lies over some edge (of some simplex in $\Sigma$) that does not occur in $\partial\Sigma=\Gamma$: write $\bar{\Delta}=\sum _i\sum_{k=1}^{j_i}(-1)^{m(i,k)}e'_{i,k}$ where, for each $i$, the $e'_{i,k}$ all lie over the same edge $e'_i$ of one of the $\sigma_j$ in $\Sigma$, and $e'_i\ne e'_s$ for $i\ne s$. Pushed forward into
$X$, the equation $\partial(\bar{\Sigma})=\bar{\Gamma}+\bar{\Delta}$ reads $\partial{\Sigma}=\Gamma+
\sum_i(\sum_{k=1}^{j_i}(-1)^{m(i,k)})e'_i$. Since $\partial{\Sigma}=\Gamma$, we see that $$\sum_{k=1}^{j_i}(-1)^{m(i,k)}=0$$ for all $i$. Now, from the displayed equation above, we have the `formal' statement
$$\bar{\gamma}(0)-\bar{\gamma}(1)=\sum_i\sum_{k=1}^{j_i}(-1)^{m(i,k)}\partial e'_{i,k}.$$
But this is also valid as a statement of equality between elements in the additive group $\R^n$ and, as group elements , $\partial e'_{i,k}=\partial e'_i$ for each $i,k$. Thus the right hand side of the last displayed equation is, as a group element, $\sum_i(\sum_{k=1}^{j_i}(-1)^{m(i,k)})\partial e'_i =0$. Hence $l=0$.
\end{proof}

Thus, since each element of $H$ is the homology class of a loop, $l:H\to\R^n$ given by $l([\gamma])=l(\gamma)$ is a well-defined homomorphism. Note that if $\gamma$ is a loop defined by return vector $v$,
then $l([\gamma])=v$.

\begin{lem}\label{f_*} $l\circ f_*=\Lambda\circ l$.
\end{lem}
\begin{proof} It follows from $f\circ p^Q=p^{\Phi(Q)}\circ\Lambda$ that if $\bar{\gamma}$ is a lift of $\gamma$ to $\R^n$, then $\Lambda\circ\bar{\gamma}$ is a lift of $f\circ\gamma$ to $\R^n$.
Thus $l([f\circ\gamma])=\Lambda\circ\bar{\gamma}(1)-\Lambda\circ\bar{\gamma}(0)=\Lambda(\bar{\gamma}(1)-\bar{\gamma}(0))=\Lambda(l([\gamma])$.
\end{proof}
Let us call a vector $v\in\R^n$ a {\em generalized return vector} for $\Phi$ if there are $v_i\in\R^n$ and $T_i\in\Omega$, $i=1,\ldots,k$, with $v=v_1+\cdots+v_k$, so that $B_0[T_i+v_i]=B_0[T_{i+1}]$ for $i=1,\ldots,k-1$,
and $B_0[T_k+v_k]=B_0[T_1]$. Let $GR=GR(\Phi)$ be collection of generalized return vectors.
It is a consequence of the next lemma that $GR(\Phi)$ is a subgroup of $\R^n$.

\begin{lem}\label{H} Let $X$ be the collared Anderson-putnam complex for $\Phi$ and let $H=H_1(X;\Z)$.
Then $l(H)=GR(\Phi)$.
\end{lem}
\begin{proof} If $\gamma$ is a loop in $X$. Then $\gamma$ is homotopic to a loop that factors as a product of paths $\gamma_i$ with each $\gamma_i$ a two-piece piecewise linear path from a prototile $\rho_{\alpha(i)}$ to an adjacent prototile $\rho_{\alpha(i+1)}$. Each $\gamma_i$ is itself homotopic rel endpoints to a product of two-piece piecewise linear paths each making an allowed transition between prototiles (this from the definition of the Anderson-Putnam complex: the glueing of prototiles along faces is by the smallest equivalence relation that includes the relation defined by the allowed transitions). Thus the displacement of the lift of $\gamma$ is a generalized return vector.

Conversely, each generalized return vector $v$ determines  a (at least one) loop $\gamma$ with
$l([\gamma])=v$.
\end{proof}

The {\em stable manifold} of $T'\in\Omega$ is the set $W^s(T'):=\{T\in\Omega:d(\Phi^k(T),\Phi^k(T'))\to0 \text{ as } k\to\infty\}$.

\begin{lem}\label{stable mfd} Suppose that $\Omega=\Omega_{\Phi}$ is an $n$-dimensional Pisot family tiling space whose $\R^n$-action has pure discrete spectrum and suppose that $T\in W^s(T')$. Then $T$ and $T'$ are densely eventually coincident.
\end{lem}
\begin{proof} Let $g:\Omega\to\hat{\T}^{md}$ be the map onto the maximal equicontinuous factor, let $\hat{F}:\hat{\T}^{md}\to\hat{\T}^{md}$ be the hyperbolic automorphism with $g\circ \Phi=\hat{F}\circ g$, and let $z\mapsto z-v$ denote the Kronecker action on $\hat{\T}^{md}$, so that $g(T-v)=g(T)-v$ for all $T\in\Omega$ and $v\in\R^n$. Suppose that $T\in W^s(T')$ but that $\emptyset\ne U\subset\R^n$ is such that $T-x\notin W^s(T'-x)$ for all $x\in U$. For fixed $x_0\in U$, let $T_k:=\Phi^k(T-x_0)$, $T'_k:=\Phi^k(T'-x_0)$, $z_k:=g(T_k)$, and $z'_k:=g(T'_k)$. Then $d(z_k,z'_k)\to 0$ as $k\to\infty$.
Choose $k_i$ with $T_{k_i}\to S$, $T'_{k_i}\to S'$; then also $z_{k_i},z'_{k_i}\to z:=g(S)=g(S')$. Since the $\R^n$-action on $\Omega$ has pure discrete spectrum, $S$ and $S'$ are proximal. Thus, given $r>0$, there is $y=y(r)\in\R^n$ so that $B_r[S-y]=B_r[S'-y]$ (see \cite{BKe}). From the local product structure on $\hat{\T}^{md}$, there is $\epsilon>0$ so that if $0<|w|<\epsilon$ and $w\in\R^n$, then $\bar{z}-w\notin W^s_{\epsilon}(\bar{z}):=\{\bar{z}':d(\hat{F}^k(\bar{z})',\hat{F}^k(\bar{z})\le\epsilon \text{ for all }k\ge 0\}$, for all $\bar{z}\in\hat{\T}^{md}$. Let $r$ be large enough so that if $\bar{T},\bar{T}'\in\Omega$ and $B_r[\bar{T}]=B_r[\bar{T}']$, then $g(\bar{T})\in W^s_{\epsilon}(g(\bar{T}'))$. Let $y=y(r)$, as above, and let $i$ be large enough so that $B_r[T_{k_i}-y-w]=B_r[T'_{k_i}-y]$ with $|w|<\epsilon$. We may also take $i$ large enough so that $x_0+\Lambda^{-k_i}y\in U$.
In $\hat{\T}^{md}$ it is the case that if $\bar{z}\in W^s_{\epsilon}(\bar{z}')$, then $\bar{z}-v\in W^s_{\epsilon}(\bar{z}'-v)$ for all $v\in\R^n$. Thus , we may further increase $i$ so that $z_{k_i}-y\in W^s_{\epsilon}(z'_{k_i}-y)$. But $z_{k_i}-y-w\in W^s_{\epsilon}(z'_{k_i}-y)$ and $|w|<\epsilon$. Thus $w=0$ so that $T_{k_i}-y\in W^s(T'_{k_i}-y)$, and then $T-x\in W^s(T'-x)$ for $x:=x_0+\Lambda^{-k_i}y\in U$, contradicting the choice of $U$. Thus $T-x\in W^s(T'-x)$ for a dense set of $x\in\R^n$ and it only remains to observe that $T$ is eventually coincident with $T'$ at $x$ if and only if $T\in W^s(T')$ (see, for example \cite{BO}).
\end{proof}

\begin{pro}\label{return} Suppose that the $\R^n$-action on $\Omega$ has pure discrete spectrum and $T\in\Omega$. Then $$\{v\in\R^n:T-v \text{ is densely eventually coincident with }T\}=\cup_{k\in\Z}\Lambda^k\mathcal{R}.$$
\end{pro}
\begin{proof} We have already observed that  $$\{v\in\R^n:T-v \text{ is densely eventually coincident with }T\}\subset\cup_{k\in\Z}\Lambda^k\mathcal{R}.$$
Suppose that $v\in\mathcal{R}$ and $k\in\Z$. There is then $\tau\in\Phi^{-k}(T)-v\cap\Phi^{-k}(T)$. Pick $x$ in the interior of the support of $\tau$. Then $\Phi^{-k}(T)-v-x\in W^s(\Phi^{-k}(T)-x)$. It follows that $T-\Lambda^kv-\Lambda^kx\in W^s(T-\Lambda^kx)$ so that, by Lemma \ref{stable mfd}, $T-\Lambda^kv-\Lambda^kx$ is densely eventually coincident with $T-\Lambda^kx$. Thus $T-\Lambda^kv$ is densely eventually coincident with $T$.
\end{proof}

\begin{pro}\label{GR} Suppose that $\Omega_{\Phi}$ is an $n$-dimensional Pisot family tiling space whose $\R^n$-action has pure discrete spectrum. Then for any admissible tiling $S$ of $\R^n$ for  $\Phi$, and any $v\in GR(\Phi)$, $S$ and $S-v$ are densely eventually coincident.
\end{pro}
\begin{proof} Let $\sim_{dc}$ be the densely eventually coincident relation ($T\sim_{dc} T' \Leftrightarrow T-x$ is eventually coincident with $T'-x$ for a dense set of $x\in\R^n$). Since eventual coincidence at a point is an open condition, $\sim_{dc}$ is an equivalence relation. Suppose that $T_i\in\Omega$ and $v_i\in\R^n$, $i=1,\ldots, k$, are such that $B_0[T_i +v_i]=B_0[T_{i+1}]$, for $i=1,\ldots,k-1$, and
$B_0[T_k+v_k]=T_1$. From the above lemma, $T_{i+1}\sim_{dc} T_i+v_i$, for $i=1,\ldots,k$, with $T_{k+1}:=T_k+v_k$, and $T_1\sim_{dc} T_{k+1}$. Clearly, $T\sim_{dc} T'\Rightarrow T-w\sim_{dc} T'-w$, so we have $T_1+v\sim_{dc} T_{k+1}$, with $v:=v_1+\cdots+v_k$. Thus $T_1+v\sim_{dc} T_1$. Now let $T$ be any element of $\Omega$ and let $T_1$ and $v$ be as above. Let $w\in\R^n$ be so that $B_0[T-w] =B_0[T_1]$. Then $T-w\sim_{dc} T_1 \sim_{dc} T_1+v$. But also, $T-w-v\sim_{dc} T_1-v$, so $T-v\sim_{dc} T$. That is, $T-v\sim_{dc} T$ for all $T\in\Omega$ and all $v\in GR(\Phi)$.

Now suppose that $S$ is admissible for $\Phi$ and let $v\in GR(\Phi)$. Let $\sigma,\tau\in S$ be such that $U=U(\sigma,\tau):=int(spt(\sigma)\cap(int(spt(\tau))-v)\ne\emptyset$. Since $S$ is admissible, there are $T_1,\ldots,T_k\in\Omega_{\Phi}$ with $\sigma\in T_1$, $\tau\in T_2$, and $T_{i+1}\cap T_i\ne\emptyset$ for $i=1,\ldots,k-1$. Restricting $\sim_{dc}$ to $U$ (that is, by $T\sim_UT'$ we mean $T-x$ is eventually coincident with $T'-x$ for a dense set of $x\in U$) we have: $S\sim_U T_1\sim_U T_1+v\sim_U T_2+v\sim_U T_3+v\sim_U\cdots\sim T_k+v\sim_U S+v$. As $\cup_{\sigma,\tau}U(\sigma,\tau)$ is dense in $\R^n$, we have $S\sim_{dc} S+v$.
\end{proof}

\begin{ex}({\bf Table substitution}) The Table substitution $\Psi$ has inflation $\Lambda=2I$ and prototiles $\rho_1=[0,2]\times[0,1]$, $\rho_2=[0,1]\times[0,2]$ that substitute as pictured. Let $Q=\{\rho_1\}$ and $v=(1,0)$.
\begin{figure}[h]
\epsfysize=1.3truein\epsfbox{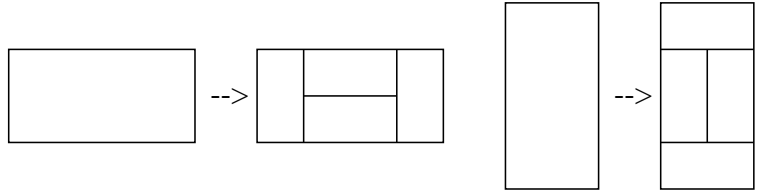}
\caption{Subdivision  for the  table}
\end{figure}
Then $\bar{Q}$ is admissible and $v$ is a return vector, but $\bar{Q}$ and $\bar{Q}-v$ are easily seen to be nowhere eventually coincident. Thus the $\R^2$-action on $\Omega_{\Psi}$ does not have pure discrete spectrum.
\end{ex}

\begin{pro}\label{eigenvalues} Suppose that $\Omega_{\Phi}$ is an $n$-dimensional substitution tiling of $(m,d)$-Pisot family type whose $\R^n$ action has pure discrete spectrum. If the product of the nonzero eigenvalues of $f_*:H_1(X;\Q)\to H_1(X;\Q)$ is $\pm 1$, then
$GR(\Phi)=\cup_{k\in\Z}\Lambda^k\mathcal{R}$.
\end{pro}
\begin{proof}
Suppose that $u\in GR(\Phi)$. If $T\in\Omega$ then $T$ and $T-u$ are densely eventually coincident by Proposition \ref{GR}. By Proposition \ref{return}, $u\in\cup_{k\in\Z}\Lambda^kv$.

Conversely, if the product of the nonzero eigenvalues of $f_*:H_1(X;\Q)\to H_1(X;\Q)$ is $\pm 1$,
there is $j\in\N$ and a subgroup $H_{ER}$ of the free part $H_{free}$ of $H_1(X;\Z)$ so that $f_*^j(H_{free})=H_{ER}$ and
$f_*|_{H_{ER}}:H_{ER}\to H_{ER}$ is invertible. Let $v\in\mathcal{R}$ and let $[\gamma]\in H_1(X,\Z)$
be so that $l([\gamma])=v$. There are then $[\alpha]\in H_{ER}$ and $[\beta]\in H_1(X;\Z)$ so that
$[\gamma]=[\alpha]+[\beta]$ and $f_*^j([\beta])$ is a torsion element. It follows from Lemma \ref{f_*} that $l([\beta])=0$
and hence $l([\alpha])=v$. We have $\Lambda^kv=\Lambda^kl([\alpha])=l((f_*|_{H_{ER}})^k([\alpha]))\in GR(\Phi)$, by Lemmas \ref{f_*} and \ref{H}.
\end{proof}

Even when the $\R^n$ action on an $n$-dimensional Pisot family tiling space has pure discrete spectrum, it is not necessarily the case that $GR(\Phi)=\cup_{k\in\Z}\Lambda^k\mathcal{R}$. For example, the Period Doubling substitution with tiles of unit length and inflation factor 2 (presented symbolically by $a\mapsto ab$, $b\mapsto aa$) has $GR=\Z$ and $\cup_{k\in\Z}\Lambda^k\mathcal{R}=\Z[1/2]$. To obtain equality, it is necessary that the eigenvalues of $\Lambda$ be algebraic units.

Substitutions satisfying the first set of hypotheses in the following corollary are called {\em unimodular, homological Pisot} and it is conjectured (a version of the `Pisot Conjecture') that their
associated $\R^n$-actions always have pure discrete spectrum (see \cite{BBJS}).

\begin{cor} Suppose that $\Phi$ is an $n$-dimensional substitution of $(m,d)$-Pisot family type with $dim(\check{H}^1(\Omega_{\Phi};\Q))=md$ and also that the eigenvalues of the expansion $\Lambda$ are algebraic units. If the $\R^n$-action on $\Omega_{\Phi}$ has pure discrete spectrum then $GR(\Phi)=\cup_{k\in\Z}\Lambda^k\mathcal{R}$.
\end{cor}
\begin{proof} Since $\il f:X\to X$ is homeomorphic with $\Omega_{\Phi}$, $\check{H}^1(\Omega_{\Phi};\Q)$ is isomorphic with $\dl f^*:H^1(X;\Q)\to H^1(X;\Q)$. This direct limit is isomorphic with the eventual range of $f^*$ and hence has dimension $md$. It is proved in \cite{BKS} that $g^*:\check{H}^1(\mathbb{T}^{md};\Q)\to\check{H}^1(\Omega_{\Phi};\Q)$  is an injection and that the eigenvalues of $F^*:\check{H}^1(\mathbb{T}^{md};\Q)\to\check{H}^1(\mathbb{T}^{md};\Q)$ are precisely the eigenvalues of $\Lambda$, together with all their algebraic conjugates, all of multiplicity $m$ ($\hat{\T}^{md}=\T^{md}$ and $\hat{F}=F$ since the eigenvalues of $\Lambda$ are units). Thus $g^*$ is surjective as well and $\hat{f}^*:\check{H}^1(\il f;\Q)\to \check{H}^1(\il f;\Q)$, being conjugate with $\Phi^*$ and hence with $F^*$, has those same eigenvalues. Since $\hat{f}^*$ is conjugate with $f^*$ restricted to its eventual range, and the latter is (by the Universal Coefficient Theorem) just the transpose (dual) of $f_*:H^1(X;\Q)\to H^1(X;\Q)$, restricted to its eventual range, we conclude that the eigenvalues of $f_*$ are precisely those of $\Lambda$, together with their algebraic conjugates, all of multiplicity $m$. As these eigenvalues are assumed to be units, the product of the nonzero eigenvalues of $f_*$ is $\pm1$ and Proposition \ref{eigenvalues} applies.
\end{proof}

\section{Applications}\label{apps}

Given a substitution $\Phi$ and tilings $S$ and $S'$ by prototiles for $\Phi$, an {\em overlap} for $S$ and $S'$ is a pair of tiles $\{\tau,\tau'\}$ with $\tau\in S$, $\tau'\in S'$,
and $\mathring{\tau}\cap\mathring{\tau}'\ne\emptyset$. The overlap $\{\tau,\tau'\}$ {\em leads to coincidence} if there is $x\in\mathring{\tau}\cap\mathring{\tau}'$ so that $S$ and $S'$ are eventually coincident at $x$. It is clear that if all pairs $\{\tau,\tau'\}$ that occur as overlaps for $\Phi^k(S)$
and $\Phi^k(S')$, $k\in\N$, lead to coincidence, then $S$ and $S'$ are densely eventually coincident. Pisot family substitution tilings have the Meyer property (see \cite{LS}) from which it follows that if $Q$ is an admissible patch for $\Phi$ that tiles periodically and $v\in GR(\Phi)\cup\cup_{k\in\Z}\Lambda^k\mathcal{R}$, then the collection of all translation equivalence classes of
overlaps for $\Phi^k(\bar{Q}),\Phi^k(\bar{Q}-v)$, $k\in\N$, is finite. In the examples of this section we implement Theorem \ref{main theorem} by finding appropriate $Q$ and $v$ so that each of the finitely many types of overlap that occur for $\Phi^k(\bar{Q}),\Phi^k(\bar{Q}-v)$, $k\in\N$, leads to coincidence.

\begin{ex} ({\bf Octagonal substitution})

The (undecorated) octagonal substitution is a 2-dimensional Pisot family substitution with inflation $\Lambda=(1+\sqrt{2})I$. There are 20 prototiles: four unmarked rhombi, $\rho_i=r^i(\rho_0)$, $i=0,\ldots3$, with $r^4(\rho_0)=\rho_0$, $r$ being rotation through $\pi/4$, and sixteen marked isosceles right triangles $\tau_i=r^i(\tau_0),\tau'_i=r^i(\tau'_0)$,
$i=0,\ldots,7$, with $\tau_i$ and $\tau'_i$ having the same support but bearing different marks. Let $\rho_0$ have vertices $(0,0),(1,0),(\sqrt{2}/2,\sqrt{2}/2)$, and $(1+\sqrt{2}/2,\sqrt{2}/2)$, and let $\tau_0$ and $\tau'_0$ have vertices $(0,0),(1,0)$, and $(1,1)$. The octagonal substitution, $\Phi_O$, is described in the following figure.
\begin{figure}[h]
\epsfysize=1.3truein\epsfbox{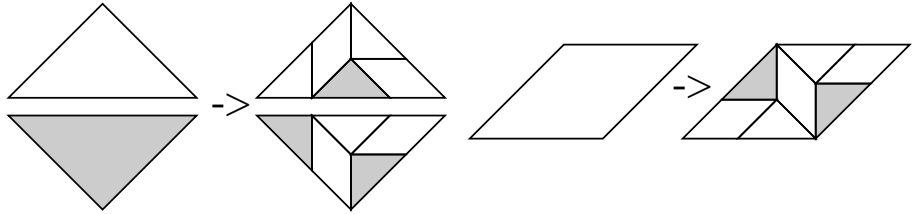}
\caption{Octagon subdivision rule. The large shaded and unshaded triangles are $\tau'_7$ and $\tau_3$, resp.}
\end{figure}
Let $Q$ be the unit square patch $Q=\{\tau_0,\tau'_4+(1,1)\}$ and let $v=(1-\sqrt{2}/2,\sqrt{2}/2)$.
Then $v$ is completely rationally independent of $\{(1,0),(0,1)\}$ and the following sequence of pictures proves that $\bar{Q}$ and $\bar{Q}-v$ are densely eventually coincident.
\begin{figure}[h]
\epsfysize=3.25truein\epsfbox{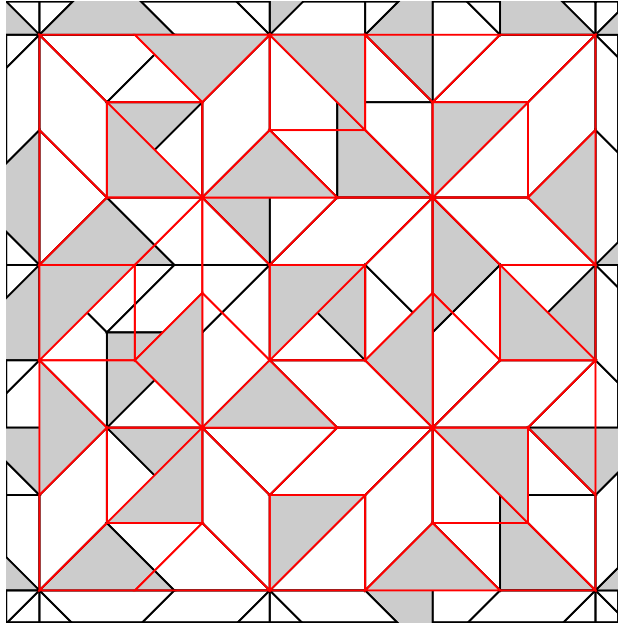}
\caption{Red overlay of the patch $\Phi_O^2(Q)$ with background $\Phi_O^2(\bar{Q}-v)$.}
\end{figure}
Indeed, we see
that each overlap that occurs for $\Phi_O^3(\bar{Q}),\Phi_O^3(\bar{Q}-v)$ already occurs (up to translation) for $\Phi_O^2(\bar{Q}),\Phi_O^2(\bar{Q}-v)$ and leads to coincidence in $\Phi_O^4(\bar{Q}),\Phi_O^4(\bar{Q}-v)$. Thus the $\R^2$-action on $\Omega_{\Phi_O}$ has pure discrete spectrum. (This conclusion is well known: octagonal tilings can be obtained by a cut-and-project method with canonical window and such a provenance always guarantees pure discrete spectrum - see, for example, \cite{BM}.)
\begin{figure}[h]
\epsfysize=3.25truein\epsfbox{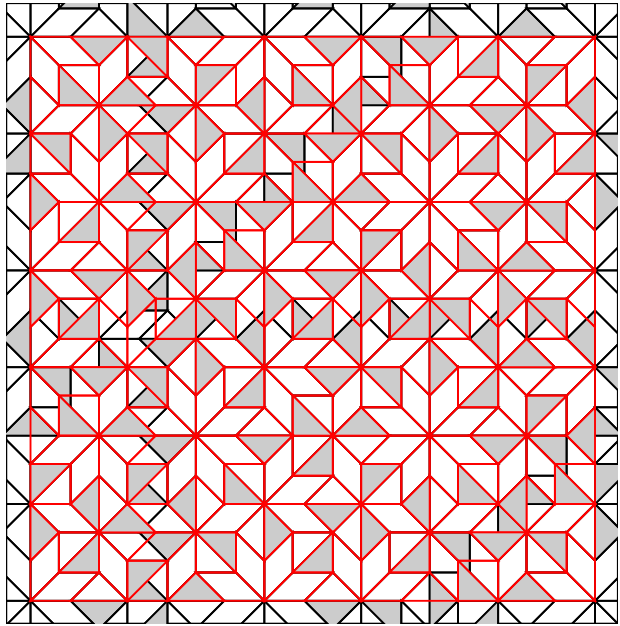}
\caption{$\Phi_O^3(Q)$ and $\Phi_O^3(\bar{Q}-v)$.}
\end{figure}
\begin{figure}[h]
\epsfysize=3.25truein\epsfbox{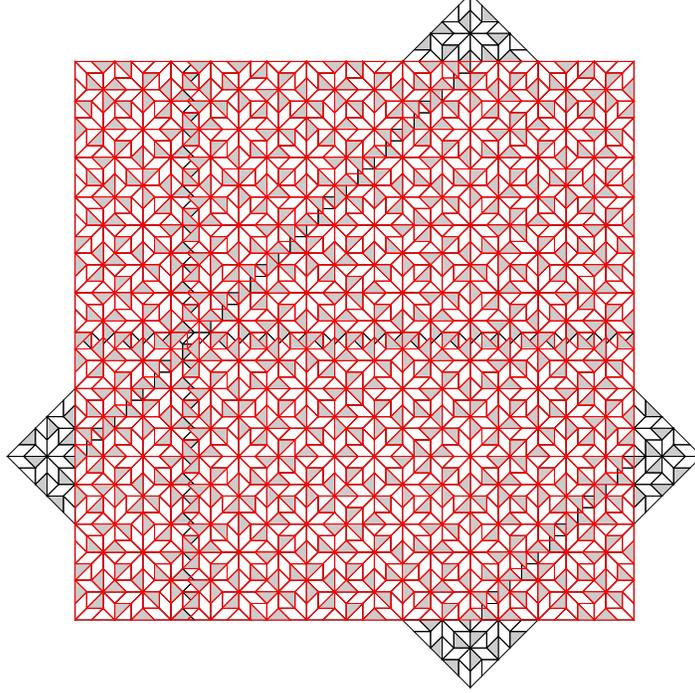}
\caption{$\Phi_O^4(Q)$ and $\Phi_O^4(\bar{Q}-v)$.}
\end{figure}
\end{ex}

\newpage
\begin{ex} ({\bf Arnoux-Rauzy substitutions})
\end{ex}

 V. Berth\'{e}, T. Jolivet, and A. Siegel have obtained the result of this subsection by different methods (\cite{BJS}).

We will say that words $w$ and $w'$ in the alphabet $\mathcal{A}=\{1,\ldots,d\}$ are {\em cyclically equivalent} if there are non-degenerate words $u,v$ with $w=uv$ and $w'=vu$. Such a cyclic equivalence is {\em non-parallel} if the abelianizations $[u]$ and $[v]$ are non-parallel vectors.
For each $i\in\mathcal{A}$, let $\sigma_i:\mathcal{A}\to\mathcal{A}^*$ be the substitution defined by $\sigma_i(j)=ji$, for $j\ne i$, and $\sigma_i(i)=i$.

\begin{lem}\label{lem:one-two} Suppose that $u^i\in\mathcal{A}^*$, $i = 1, \dots , d$, are cyclically equivalent,
$k \in \{ 2, \dots , d \}$ and $j \in \{ 1, \dots , k \}$. Then
$\sigma_j(u^1) = v^1, \dots , \sigma_j(u^k) = v^k$,
$\sigma_j((k+1)u^{k+1}) = (k+1)v^{k+1}, \dots , \sigma_j(du^d) = dv^d$, with $v^i$, $i = 1, \dots , d$,
cyclically equivalent.
\end{lem}
\begin{proof} Let $k \in \{ 2, \dots , d \}$, $j \in \{ 1, \dots , k \}$ and $l \in \{ k+1, \dots , d \}$.
Then $\sigma_j(lu^l) = lj\sigma_j(u^l) = lj(w^lj) = l(jw^l)j$. Therefore, $v^i$, $i \in \{ 1, \dots , k \}$,
and $w^lj$, $l = k+1, \dots , d$, are cyclically equivalent, and hence $v^i$, $i \in \{ 1, \dots , k \}$,
and $v^l = jw^l$, $l = k+1, \dots , d$, are cyclically equivalent.
\end{proof}

\begin{lem}\label{lem:three} Suppose that $u^i$, $i = 1, \dots , d$, are cyclically equivalent and
$k \in \{ 3, \dots , d \}$. Then $\sigma_k(u^1) = v^1, \dots , \sigma_k(u^{k-1}) = v^{k-1}$,
$\sigma_k(ku^k) = v^k$
$\sigma_k((k+1)u^{k+1}) = (k+1)v^{k+1}, \dots , \sigma_k(du^d) = dv^d$, with $v^i$, $i = 1, \dots , d$,
cyclically equivalent.
\end{lem}
\begin{proof} Let $k \in \{ 3, \dots , d \}$ and $l \in \{ k+1, \dots , d \}$.
Then $\sigma_k(ku^k) = k\sigma_k(u^k) = k(w^kk) = (kw^k)k$ and
$\sigma_k(lu^l) = lk\sigma_k(u^l) = lk(w^lk) = l(kw^l)k$. Therefore, $v^i$, $i \in \{ 1, \dots , k-1 \}$,
$w^kk$ and $w^lk$, $l = k+1, \dots , d$, are cyclically equivalent, and hence
$v^i$, $i \in \{ 1, \dots , k-1 \}$, $v^k = kw^k$ and $v^l = kw^l$, $l = k+1, \dots , d$, are cyclically equivalent.
\end{proof}

Given a word $w=w_1w_2\cdots w_k\in\mathcal{A}^*$, let $\sigma_w:=\sigma_{w_1}\circ\cdots\circ\sigma_{w_k}$. A substitution $\phi:\mathcal{A}\to\mathcal{A}^*$ is {\em Arnoux-Rauzy} if $\phi=\sigma_w$ for some $w$ with $[w]$ strictly positive.

\begin{pro}\label{pro:cyc-equiv} If $\phi = \sigma_w$ is Arnoux-Rauzy on $d$ letters, then
there are prefixes  $P_i$ of the fixed words $\phi^{\infty}(i) = P_i \dots$, $i = 1, \dots , d$ so that there is a non-parallel cyclic equivalence  between $P_i$ and $P_j$ for $i\ne j\in\{1, \dots , d\}$.
\end{pro}
\begin{proof} By symmetry of $\sigma_i$'s we may assume, after renumbering, that $w = YdX_{d-2}d-1 \dots 3X_121^n$, with $X_j$
having only 1's, $\dots$, $(j+1)$'s.

Note that
\begin{eqnarray*}
\sigma_{21^n}(1) &=& {\bf 12}\\
\sigma_{21^n}(2) &=& {\bf 21}212\dots\\
\sigma_{21^n}(3) &=& 3{\bf 21}2\dots\\
               &\vdots&   \\
\sigma_{21^n}(d) &=& d{\bf 21}2\dots .
\end{eqnarray*}
By Lemma \ref{lem:one-two} (for $k = 2$) we have that
\begin{eqnarray*}
\sigma_{X_1}\sigma_{21^n}(1) &=& v^1_1\\
\sigma_{X_1}\sigma_{21^n}(2) &=& v^2_1\\
\sigma_{X_1}\sigma_{21^n}(3) &=& 3v^3_1\dots\\
               &\vdots&   \\
\sigma_{X_1}\sigma_{21^n}(d) &=& dv^d_1\dots .
\end{eqnarray*}
with $v^i_1$, $i = 1, \dots , d$, cyclically equivalent.

By Lemma \ref{lem:three} (for $k = 3$) we have that
\begin{eqnarray*}
\sigma_3\sigma_{X_1}\sigma_{21^n}(1) &=& v^1_2\\
\sigma_3\sigma_{X_1}\sigma_{21^n}(2) &=& v^2_2\\
\sigma_3\sigma_{X_1}\sigma_{21^n}(3) &=& v^3_2\dots\\
\sigma_3\sigma_{X_1}\sigma_{21^n}(4) &=& 4v^4_2\dots\\
               &\vdots&   \\
\sigma_3\sigma_{X_1}\sigma_{21^n}(d) &=& dv^d_2\dots .
\end{eqnarray*}
with $v^i_2$, $i = 1, \dots , d$, cyclically equivalent, etc.

Let $P_i:=\sigma_Y(v^i_{2d})$. Then $P_i$ and $P_j$ are cylically equivalent for $i\ne j$.
If $P_1=uv$ with $u,v$ non-degenerate, then $M_{\phi}[1]=[\phi(1)]=[P_1]=[P_i]=[u]+[v]$.
So if $[v]=t[u]$, then $\frac{1}{1+t}[1]=M^{-1}_{\phi}[u]$. But $M^{-1}_{\phi}[u]$ is in $\Z^d$ and
$\frac{1}{1+t}[1]$ is not. Thus the cyclic equivalences are non-parallel.

\end{proof}

If $\phi$ is an Arnoux-Rauzy substitution on $d$ letters with Pisot inflation factor (that is, the Perron-Frobenius eigenvalue of $\lambda$ of $M_{\phi}$ is Pisot) we'll say that $\phi$ is {\em Pisot Arnoux-Rauzy} and {\em irreducible Pisot Arnoux-Rauzy} if $M_{\phi}$ is irreducible. It is shown in \cite{AI} that all Arnoux-Rauzy substitutions on three letters are irreducible Pisot. If $\phi$ is Pisot Arnoux-Rauzy, the maximal equicontinuous factor for the $\R$-action on the associated 1-dimensional tiling space $\Omega_{\Phi}$ is an $\R$-action on the $m$-torus $\mathbb{T}^m=\R^m$, $m=deg(\lambda)$ (this is because $det(M_{\phi})=1$, so $\lambda$ is a unit). The maximal equicontinuous factor map $g:\Omega_{\Phi}\to\mathbb{T}^m$
semi-conjugates $\Phi$ with a hyperbolic toral automorphism $F:\mathbb{T}^m\to\mathbb{T}^m$.

\begin{lem}\label{same g} If $\phi$ is a Pisot Arnoux-Rauzy substitution, $T$ and $T'$ in $\Omega_{\Phi}$ are tilings that have a common vertex, and $n_k\to\infty$ is such that $\Phi^{n_k}(T)\to S\in\Omega_{\phi}$ and $\Phi^{n_k}(T')\to S'\in\Omega_{\Phi}$, then $g(S)=g(S')$.
\end{lem}
\begin{proof}
Suppose that  $T$ and $T'$ have a common vertex at $t_0$. Then $\Phi(T)$ and $\Phi(T')$ share a tile
(of type 1) with right vertex at $\lambda t_0$. Then, for $t$ slightly bigger than $t_0$, $T-t$ is in the $\Phi$-stable manifold of
$T'-t$ and $g(T-t)=g(T)-\tilde{t}$ is in the $F$-stable manifold of $g(T'-t)=g(T')-\tilde{t}$. But then
$g(T)$ is in the same $F$-stable manifold as $g(T')$ so that $g(S)=g(\lim_{k\to\infty} \Phi^{n_k}(T))=\lim_{k\to\infty}F^{n_k}(g(T))=\lim_{k\to\infty}F^{n_k}(g(T'))=g(\lim_{k\to\infty} \Phi^{n_k}(T'))=g(S')$.
\end{proof}

Let $cr=cr(\phi)$ be the {\em coincidence rank} of the Pisot substitution $\phi$:
$cr:=max\sharp\{T_1,\ldots,T_m:g(T_i)=g(T_j)\text{ and }T_i\cap T_i=\emptyset \text{ for }i\ne j\}$. The coincidence rank is finite and $g$ is a.e. $cr$-to-1 (see \cite{BK} or \cite{BKe}). Furthermore, if $g(T_i)=g(T_j)\text{ and }T_i\cap T_i=\emptyset$, then $T_i$ and $T_j$ are nowhere eventually coincident (\cite{BKe}).
Let $T_1,\ldots,T_{cr}\in\Omega_{\phi}$ be tilings with $g(T_i)=g(T_j)$ and $T_i$ nowhere eventually coincident with $T_j$ for $i\ne j$, let $V_i$ be the collection of vertices of $T_i$,  let $V:=\cup_{i=1}^{cr}V_i$, and write $V=\{v_n:v_n<v_{n+1}, n\in\Z\}$. Note that if $\phi$ is Arnoux-Rauzy, then $V_i\cap V_j=\emptyset$ for $i\ne j$ (if $T$ and $T'$ share a vertex, then $\Phi(T)$ and $\Phi(T')$ share a tile of type 1). For each $n\in\Z$ and $i\in\{1,\ldots,cr\}$, let $\tau_{n(i)}$ be the tile of $T_i$ whose support meets $(v_n,v_{n+1})$. The pair $C_n:=([v_n,v_{n+1}],\{\tau_{n(1)},\ldots,\tau_{n(cr)}\})$ is called a {\em configuration}. Configurations $C_n$ and $C_m$ {\em have the same type} if there is $t\in\R$ so that $[v_n,v_{n+1}]-t=[v_m,v_{m+1}]$ and $\{\tau_{n(1)}-t,\ldots,\tau_{n(cr)}-t\}= \{\tau_{m(1)},\ldots,\tau_{m(cr)}\}$. By a result of \cite{BKe}, there are only finitely many types of configurations.

\begin{lem}\label{lem:configurations} If $\phi$ is Arnoux-Rauzy, there are $m<n$ so that the configurations $C_n$ and $C_m$ have the same type, but $C_{n+1}$ and $C_{m+1}$ have different types.
\end{lem}
\begin{proof}
Otherwise, there is $L>0$ so that $C_n$ and $C_{n+L}$ have the same type for all $n\in\Z$.
For $n\in\Z$ let $i(n)\in\{1,\ldots.cr\}$ be (uniquely) defined by: $v_n$ is a vertex of a tile in $T_{i(n)}$. There are then $n\in\Z$ and $k\in\N$ so that $i(n)=i(n+kL)$. But then $i(n+jkL)=i(n)$ for all $j\in\Z$ and the tiling $T_{i(n)}$ is periodic.
\end{proof}

For an Arnoux-Rauzy substitution $\phi$ on $d$ letters, let $S_i\in\Omega_{\Phi}$ denote the tiling fixed by $\Phi$ containing a tile of type $i$ with the origin at its left endpoint.

\begin{lem}\label{lem:coincidence} If $\phi$ is a Pisot Arnoux-Rauzy substitution, then there are $i\ne j$ so that the fixed tilings $S_i,S_j\in\Omega_{\Phi}$ are densely eventually coincident.
\end{lem}
\begin{proof}
By Lemma \ref{lem:configurations}, there are $m<n$ so that the configurations $C_n$ and $C_m$ have the same type, but $C_{n+1}$ and $C_{m+1}$ have different types. Let $i$ be such that vertex $v_{n+1}$ is the initial vertex of a tile $\tau$ of type $i$ in
$T_{i(n+1)}$ and let $j$ be such that vertex $v_{m+1}$ is the initial vertex of a tile $\tau'$ of type $j$ in $T_{i(m+1)}$. Note that $\{\tau_{n+1(1)}-v,\ldots,\tau_{n+1(cr)}-v\}\setminus(\{\tau-v\})=
\{\tau_{m+1(1)},\ldots,\tau_{m+1(cr)}\}\setminus \{\tau'\}$, with $v:=v_{n+1}-v_{m+1}$. Let $b:=min\{v_{m+2},v_{n+2}-v\}$. Suppose that there are $t_0$ and $\epsilon>0$ so that $(t_0-\epsilon,t_0+\epsilon)\subset (v_{m+1},b)$ and $\tau_{m+1(1)},\ldots,\tau_{m+1(cr)},\tau-v$ are pairwise nowhere eventually coincident on
$(t_0-\epsilon,t_0+\epsilon)$. There is then $n_k\to\infty$ so that $\Phi^{n_k}(T_j-t_0)\to T'_j\in\Omega_{\Phi}$ for $j=1,\ldots,cr$ and $\Phi^{n_k}(T_{i(n+1)}-v-t_0)\to T'_{cr+1}\in\Omega_{\Phi}$.
The tilings $T'_j$, $j=1,\ldots,cr+1$, are then pairwise nowhere  eventually coincident and they all have the same image under $g$, by Lemma \ref{same g}. This violates the definition of $cr$.
As the elements of $\{\tau_{m+1(1)},\ldots,\tau_{m+1(cr)}\}$, and of $(\{\tau_{m+1(1)},\ldots,\tau_{m+1(cr)}\}\setminus\{\tau'\})\cup \{\tau-v\}$ are pairwise nowhere eventually coincident, it must be the case that there is $t\in(t_0-\epsilon,t_0+\epsilon)$ with $\tau' $ and $\tau-v$ eventually coincident at $t$. That is, $\tau'$ and $\tau-v$ are densely eventually coincident on $(v_{m+1},b)$. This means
that there is $\delta>0$ so that $S_i$ and $S_j$ are densely eventually coincident on $(0,\delta)$.
As $\Phi^k(S_i)=S_i$ and $\Phi^k(S_j)=S_j$ for all $k$, we have that $S_i$ and $S_j$ are densely
eventually coincident on $(0,\lambda^k\delta)$ for all $k$, so $S_i$ and $S_j$ are densely eventually coincident.
\end{proof}

\begin{theorem}\label{thm:AR one-to-one} If $\phi$ is an irreducible Pisot Arnoux-Rauzy substitution, then the $\R$-action on $\Omega_{\Phi}$ has pure discrete spectrum.
\end{theorem}
\begin{proof}
Let $S_i$ and $S_j$, $i\ne j$, be the fixed tilings in $\Omega_{\Phi}$ that are densely eventually coincident (Lemma \ref{lem:coincidence}), and let $Q_i$ and $Q_j$ be their initial
patches, corresponding to the prefixes $P_i$ and $P_j$ that are non-parallel cyclically equivalent (Proposition \ref{pro:cyc-equiv}). Let $u$ and $v$ be non-degenerate words with non-parallel abelianizations  so that $P_i=uv$ and $P_j=vu$, let $l_u$ be the length of the patch corresponding to $u$, and let $l$ be the length of $Q_i$. Suppose that $rl_u-sl=0$ for integers $r,s$. The length of 
a tile of type $i$ is given by $\langle [i],\omega\rangle$ where $\omega=(\omega_1,\ldots,\omega_d)$ is a positive left Perron eigenvector of $M_{\phi}$.
We have $0=rl_u-sl=r\langle[u],\omega\rangle-s\langle[u+v],\omega\rangle=\langle(r-s)[u]-s[v],\omega\rangle$. But the $\omega_i$ are independent over $\mathbb{Q}$ (by irreducibility), so the only integer vector orthogonal to $\omega$ is $\bf{0}$. Thus $(r-s)[u]=s[v]$, and since $[u]$ and $[v]$ are not parallel,
$r=0=s$. Thus, $l_u$ is irrationally related to $l$. We have that $\bar{Q}_i$ and $\bar{Q}_j=\bar{Q}_i-l$ are densely eventually coincident. By Theorem \ref{main theorem}, the $\R$-action on $\Omega_{\Phi}$ has pure discrete spectrum.
\end{proof}

\begin{ex} ({\bf Arbitrary compositions of the Rauzy and modified Rauzy substitutions})

Let us consider the Rauzy substitution $\tau_1$: $1 \to 12$, $2 \to 13$ and $3 \to 1$,
and modified Rauzy substitutions defined as follows, $\tau_2$: $1 \to 12$, $2 \to 31$ and $3 \to 1$;
$\tau_3$: $1 \to 21$, $2 \to 13$ and $3 \to 1$; and $\tau_4$: $1 \to 21$, $2 \to 31$ and $3 \to 1$.
It is easy to check that applying $\tau_1$, $\tau_2$, $\tau_3$ or $\tau_4$ to any pair in the list
below results in a pair that can be factored as a product of pairs in the list below (or of the duals $(v,u)$ of pairs $(u,v)$ in the list) and coincidences.

List of irreducible pairs:
(12, 21), (13, 31), (123, 231), (321, 132), (213, 312), (1123, 3112), (3211, 2113), (1213, 3112), (3121, 2113),
(1213, 3121), (1231, 3112), (1321, 2113), (2131, 3112), (1312, 2113), (11231, 31112), (12123, 23112), (32121, 21132),
(12311, 31112), (11321, 21113), (121213, 311212), (312121, 212113), (121123, 311212), (321121, 212113),
(121213, 312112), (312121, 211213), (121231, 231112), (132121, 211132), (1121231, 3112112), (1321211, 2112113), (12121231, 23111212), (13212121, 21211132)

Note that for any $i,j\in\{1,2,3,4\}$ and any $a,b\in\{1,2,3\}$, $(\tau_i\circ\tau_j(a),\tau_i\circ\tau_j(b))$ is either of the form $(uk\ldots,vk\ldots)$ or $(\ldots ku,\ldots kv)$ for some $k\in\{1,2,3\}$ and (possibly empty) words $u,v$ with $[u]=[v]$.
Thus if $\tau$ is any finite composition of the $\tau_i$'s, the balanced pair algorithm for $\tau$ applied to $(12,21)$ terminates with coincidence. Since such $\tau$ are Pisot with irreducible incidence matrix, the following is a consequence of Corollary \ref{(ab,ba)}.

\begin{theorem} If $\tau$ is any finite composition of Rauzy and modified Rauzy substitutions, the $\R$-action on $\Omega_{\tau}$ has pure discrete spectrum.
\end{theorem}
\end{ex}


\begin{thebibliography}{99}

\bibitem[A]{Aus} J.\ Auslander, Minimal flows and their extensions,
  {\em North-Holland Mathematical Studies}, vol. 153, North-Holland,
  Amsterdam, New York, Oxford, and Tokyo, (1988).
  
\bibitem[AI]{AI} P. Arnoux and S. Ito, Pisot Substitutions and Rauzy fractals, {\em Bull. Belg. Math.
  Soc.} {\bf8} (2001), 181-2007.

\bibitem[AL]{AL} S. Akiyama and J.-Y. Lee, Algorithm for determining pure pointedness of self-
  affine tilings, {\em Adv. Math.} {\bf 226} (2011), 2855-2883.
  
\bibitem[AP]{AP} J.E. Anderson and I.F. Putnam, Topological invariants
  for substitution tilings and their associated $C^*$-algebras, {\em
    Ergodic Theory \& Dynamical Systems} \textbf{18} (1998), 509--537.
    
\bibitem[BBJS]{BBJS} M. Barge, H. Bruin, L. Jones and L. Sadun, Homological Pisot substitutions
  and exact regularity, to  appear in \textit{Israel J. Math.}
  
\bibitem[BBK]{BBK} V.\ Baker, M.\ Barge and J.\ Kwapisz, Geometric
  realization and coincidence for reducible non-unimodular Pisot
   tiling spaces with an application to $\beta$-shifts,
     \textit{J. Instit. Fourier.}  {\bf56} (7) (2006), 2213-2248.
     
\bibitem[BJS]{BJS} V. Berth\'{e}, T. Jolivet and A. Siegel, Substitutive Arnoux-Rauzy substitutions  
  have pure discrete spectrum, 2011 preprint.

\bibitem[BK]{BK} M.\ Barge and J.\ Kwapisz, Geometric theory of
  unimodular Pisot substitutions, \textit{Amer J. Math.} {\bf128}
   (2006), 1219-1282.
   
\bibitem[BKe]{BKe} M. Barge and J. Kellendonk, Proximality and pure point spectrum for tiling
  dynamical systems, preprint.
  
\bibitem[BKS]{BKS} M. Barge, J. Kellendonk and S. Schmeiding, Cohomology of Pisot family
  tiling spaces, preprint.

\bibitem[BM]{BM} M. Baake and R. V. Moody, Weighted Dirac combs with pure point diffraction,
  {\em J. Reine Angew. Math. (Crelle)} {\bf573} (2004), 61-94; math.MG/0203030.

\bibitem[BO]{BO} M.\ Barge and C.\ Olimb, Asymptotic structure in substitution tiling spaces, 2010
  preprint.

\bibitem[BS]{BS} V.\ Berth\'{e} and A.\ Siegel, Tilings associated with
  beta-numeration and substitutions, {\em Integers: electronic journal
    of combinatorial number theory} {\bf 5} (2005), A02.
    
 \bibitem[D]{D} S.\ Dworkin, Spectral theory and X-ray diffraction, {\em J. Math. Phys.} {\bf 34}
   (1993), 2965-2967.

\bibitem[De]{Dek} F.\ M.\ Dekking, The spectrum of dynamical systems
  arising from substitutions of constant length, {\em
    Z. Wahrscheinlichkeitstheorie verw. Gebiete} {\bf 41} (1978),  221-239.

\bibitem[Fo]{Fo} N. Pytheas Fogg, Substitutions in dynamics, arithmetics and combinatorics, {\em
  Lecture notes in mathematics}, V. Berth\'{e}, S. Ferenczi, C. Mauduit and A. Siegel, eds. Springer-
   Verlag, 2002.

\bibitem[FS]{FS} D. Fretl\"oh and B. Sing, Computing modular coincidences for substitution tilings
  and point sets, {\em Discrete Comput. Geom.} {\bf37} (3) (2007), 381-407.

\bibitem[IR]{IR} S. Ito and H. Rao, Atomic surfaces, tiling and coincidence I. Irreducible case, {\em
  Israel J. Math.} {\bf153} (2006), 129-156.

\bibitem[K]{K} R.\ Kenyon, Ph.D. Thesis, Princeton University, 1990.

\bibitem[KS]{KS} R. Kenyon and B. Solomyak, On the characterization of expansion maps for
 self-affine tilings, {\em Discrete Comp. Geom.} Online, 2009.

\bibitem[L]{L} J.-Y. Lee, Substitution Delone  multisets with pure point spectrum are inter-model
  sets, {\em Journal of Geometry and Physics} {\bf57} (2007), 2263-2285.

\bibitem[LM]{LM} J.-Y. Lee and R. Moody, Lattice substitution systems and model sets, {\em
  Discrete Comput. Geom.} {\bf25} (2001), 173-201.

\bibitem[LMS]{LMS} J.-Y. Lee, R. Moody and B. Solomyak, Consequences of Pure Point Diffraction
  Spectra for Multiset Substitution Systems, {\em Discrete Comp. Geom.} {\bf29}(2003), 525-560.

\bibitem[LS1]{LS1} J.-Y. Lee and B. Solomyak, Pure point diffractive substitution Delone sets have
  the Meyer property, {\em Discrete Comp. Geom.} {\bf34} (2008), 319-338.

\bibitem[LS2]{LS} J.-Y. Lee and B. Solomyak, Pisot family self-affine tilings, discrete spectrum,
  and the Meyer property, 2010 preprint.
  
\bibitem[S1]{S1} B. Solomyak, Nonperiodicity implies unique
  composition for self-similar translationally finite tilings, {\em
    Discrete Comput. Geometry} {\bf 20} (1998), 265--279.

\bibitem[S2]{S2} B.\ Solomyak, Eigenfunctions for substitution tiling systems,\emph{ Advanced Studies in Pure Mathematics}, \textbf{49} (2007), 433--454.

\bibitem[S3]{S3} B.\ Solomyak, Dynamics of self-similar tilings, {\em
    Ergodic Theory \& Dynamical Systems} {\bf17} (1997), 695-738.
    
\bibitem[SS]{SS} V. F. Sirivent and B. Solomyak, Pure discrete spectrum for one-dimensional
  substitution systems of Pisot type, {\em Canad. Math. Bull.} {\bf45}(4) (2002) 697-710. Dedicated
   to Robert V. Moody.

\bibitem[ST]{ST} A. Siegel and J. Thuswaldner, Topological properties of Rauzy fractals, {\em
  Memoire de la SMF}, To appear.





\end{thebibliography}
\end{document}